\documentclass[a4paper,12pt]{article}
\usepackage{amsmath}
\usepackage{amsthm}
\usepackage{mathrsfs}
\usepackage{amssymb}
\usepackage{amsfonts}
\usepackage{amssymb}
\usepackage{float} 
\usepackage{booktabs}
\usepackage{authblk}
\usepackage{amsthm}
\usepackage{comment}
\newtheorem{theorem}{Theorem}
\usepackage{amsmath}
\usepackage{amsthm}
\usepackage{mathrsfs}
\usepackage{amssymb}
\usepackage{hyperref}
\usepackage{multirow}
\usepackage{multicol}
\usepackage{curves}
\usepackage[utf8]{inputenc}
\usepackage[T1]{fontenc}
\usepackage{amsmath}
\usepackage{amsfonts}
\usepackage{amssymb}
\usepackage{subcaption}
\usepackage{caption}
\usepackage{graphicx}
\usepackage[table]{xcolor}
\usepackage{tkz-graph}
\usepackage{pgfplots}
\usepackage{forest}
\usepackage{chemfig}
\usepackage{float}
\definecolor{skyblue}{RGB}{70, 130, 180}

\newtheorem{example}{Example}[section] 
\newtheorem{lemma}{Lemma}[section] 

\usepackage[nottoc,other]{tocbibind}
\usetikzlibrary{positioning}

\tikzset{
    main node/.style={circle,fill=white,draw,minimum size=0.5cm,inner sep=0pt},
}

\usepackage[a4paper, margin=2.2 cm]{geometry} % Adjust margins here

\usepackage{titlesec}
\titleformat{\section}[hang]{\normalfont\Large\bfseries}{\thesection.}{1em}{}
\title{\textbf{Modified Cubic B-spline Based Differential Quadrature Methods for Time-fractional Black-Scholes Equation}}

\author[1]{Nizamudheen V}
\author[2]{Riyasudheen TK}
\author[3]{ Noufal Asharaf}
\author[4]{ Shefeeq T}

\affil[1]{\small Department of Mathematics, Farook College (Autonomous) affiliated to  University of Calicut, Kozhikode,  673632,  India, \texttt{\small nizam@farookcollege.ac.in}}
\affil[2]{\small Department of Computational Science and Humanities, Indian Institute of Information Technology Kottayam,
Valavoor, Kottayam 686635, Kerala, India, \texttt{\small riyasudheen@iiitkottayam.ac.in}} 
\affil[3]{\small Department of Mathematics, CUSAT (Cochin University of Science and Technology), Cochin,  682022, India, \texttt{\small noufal@cusat.ac.in}}
\affil[4]{\small Department of Mathematics, Farook College (Autonomous) affiliated to University of  Calicut,  Kozhikode,  673632,  India, \texttt{\small shefeeq@farookcollege.ac.in}}

\date{  } % Use the current date; or replace with 
\begin{document}
\maketitle
\noindent Corresponding Author: Shefeeq T, shefeeq@farookcollege.ac.in 
\begin{abstract}
The time-fractional Black–Scholes equation (TFBSE) is intended to price the options for which the underlying price fluctuates within a correlated fractal transmission system. Although the TFBSE is an influential approach for grasping the long-term memory traits of financial markets, the non-local nature of fractional derivatives makes significant challenges in finding an accurate solution. We perform an efficient use of the differential quadrature method (DQM) based on modified cubic B-splines to solve the TFBSE governing European options. This paper constructs an algorithm by the combination of time fractional discretization using the finite difference method $L1$ and space discretization using the modified cubic B-spline-based differential quadrature method. Uniform meshes are considered for the discretization of both temporal and spatial domains. Theoretical stability has been established by finding an estimate for the maximum norm of the inverse operator regardless of the involvement of mesh parameters. We trigger the Neumann series theorem to obtain a uniform bound for the inverse operator under reasonable conditions on the mesh parameters. The numerical illustrations show that this implicit numerical method exhibits a fourth-order convergence in the space direction and the order $2-\alpha$ in time. Moreover, we observe an enhancement in order of spatial convergence whenever $\alpha$ tends to $0$. The results obtained are then compared with existing popular techniques to demonstrate the accuracy of modified cubic B-spline-based DQM.
\end{abstract}

\medskip
\noindent \textbf{Mathematics Subject Classification (MSC 2020):} Primary: 65M12, 26A33, 91G60; Secondary: 35R11, 91G20.

\section{Introduction}

The ever-growing derivative markets have long been an area of research interest among financial
mathematicians, policymakers, and agents. It directly yields more notable contributions to the economic system as a whole.
\begin{comment}
    In the financial market, the derivatives are used widely by traders to manage the risk and speculate on  profit. Among the common derivative securities like  options, forwards and futures, options is the most used in the market by the practitioners. Hence pricing of options needs a detailed analysis in it's theoretical and practical aspects. An option is a deal between two parties whereby the buyer has the right to purchase or deliver a particular quantity of fundamental asset to the option writer at a prefixed price, known as  strike price, on or before the expiration date. Expiration date refers to the period on which the deal expires  and all  obligations must be decided. American and European options are the  two different styles of options. In the American style  there is a possibility of exercising the option at any moment on or before expiry, while in the European style exercising the option is possible only on the expiry date. Options are mainly divided  into two categories  based on  option rights: call and put. If the option grants the holder right to purchase  the stock at predetermined price is a call while a put is the right to sell. In the current era, both financial engineers and mathematicians have been interested in the option valuation. 
\end{comment}
Historically, financial derivatives and their specialized forms originated in ancient Greek philosophy. Options are financial derivatives that are configured as contracts between two parties to make a potential transaction of an asset at a preset `strike price' prior to `expiration'. In the current financial market scenario, option trading has enriched the derivative market, specifically in virtual fashion. Quite futuristic hedging strategies are required to exercise the options
effectively. Option valuation theories estimate the value of options by assigning a price, referred to as the ‘premium’. Finding a fair value for options had been defiance for the financial world until the invention of the Black-Scholes formula  \cite{Black}. Myron Scholes and Robert C Merton \cite{RC} have, in combined work with the late Fischer Black, deduced an instigating formula later to be known as the Black-Scholes-Merton formula for pricing stock options.

In 1997, they were awarded the Nobel Prize in economics sciences for developing this conceptual framework called the Black–Scholes model (BSM). In \cite{Black}, Black and Scholes demonstrated that a second-order parabolic partial differential equation with respect to time and stock price, known as the Black–Scholes equation  (BSE), governs the value of a European option on a stock whose price follows a geometric Brownian motion with constant drift and volatility. Various numerical techniques such as Monte-Carlo simulations \cite{monte}, lattice methods \cite{bin}, finite difference methods (FDM) \cite{kadal2,riyas,aswin}, finite element methods (FEM) \cite{FEM}, finite volume
methods (FVM) \cite{valkov, wang} are exerted on the classical Black-Scholes equations and its various generalizations, due to the inadequacy of analytical methodologies. However, it is well known that the hypotheses of the classical BS equation are so optimistic that it is not consistent with the real stock movement; it fails to catch the significant movements or jumps over short time steps in the financial market.

A key drawback of the classical Black-Scholes model and its alterations explored in the literature involves integer-order derivatives. Integer-order derivatives can only address localized information around a point \cite{panas}, whereas fractional derivatives and integrals serve as a tool for the description of memory \cite{Liu}; the non-local property indicates that a complex system's state depends on all its previous states in addition to its current state. This advantage has led to the increasing popularity of fractional calculus. The fast growth of fractional calculus in recent decades has led to the use of fractional partial differential equations (FPDE) across diverse domains, including physics, fluid mechanics, biology, engineering, and finance. The fractional BS equations were established to extend the financial theory in the late 20th century. In contrast to the classical BS model, its advantage is to simulate complex scenarios in the actual market. Fractional Brownian motion replaced geometric Brownian motion as a more suitable tool for capturing the aforementioned asset price behaviors. The fractional derivative operators and Hurst parameters \cite{Kiryakova} were conveyed into the model to fix the effect of memory in a financial system.  The significant contributions by Wyss \cite{Wyss} and Cartea et al. \cite{Cartea} increased the popularity of
fractional BS models. Rather than these models, numerous fractional Black-Scholes models have been introduced \cite{Heston, Chinwen, Wentingchen}. These models are dedicated to longer memory effects in asset price returns.

With the broad application of fractional BS equations in option pricing, there has been an increase in interest in exploring solution techniques. Since the beginning of this century, numerous research studies have been undertaken to find solutions from both analytical and numerical perspectives. Chen et al. \cite{Chen} derived an explicit closed-form analytical solution to price double barrier options by using the eigenfunction expansion method together with the Laplace transform. Jicheng Yu \cite{ji} implemented the Lie symmetry analysis in TFBSE using the invariant subspace method. A successful application of the homotopy analysis method (HAM) in TFBSE can be seen in the work of SE Fadugba \cite{SE}. Asma et al. \cite{asma} compared two methods in fractional models, one is a combination of the homotopy perturbation method (HPM), the Sumudu transform, and the He polynomials, and another is the homotopy Laplace transform perturbation method. However, in several cases, it is challenging to determine the explicit analytic solution of fractional BS equations. Therefore, various numerical methods have been developed in recent years. The methods are based on finite differences, finite elements, and the spectral approach. An RBF-based mesh-free method is developed by A. Golbabai et al. in combination with a finite difference method of order $\mathcal O(t^{2-\alpha})$ along the time. Nuugulu et al. \cite{nug} proposed a first-order implicit finite difference method for solving the constructed TFBSE. Zhang \cite{Zhang} employed another implicit finite difference method by discretizing the spatial derivative by central finite difference and L1 scheme for time derivative, having spatial order of convergence  2. In \cite{Rezaei} obtained a numerical solution based on the implicit difference scheme to the European option price with transaction costs. The work in  \cite{Krzyzanowski}  introduces a weighted FDM  for the subdiffusive Black-Scholes (B-S) model. The mixed alternate segment Crank-Nicolson (MASC-N)  scheme is applied by X. Yang\cite{Yang}. This scheme has a spatial order of convergence of 2. T Akram et al. \cite{tay} generalized the concept of the B-spline collocation method using extended cubic B-splines (ECBS). Zhaowei et al. \cite{Zhao} presented a compact quadratic spline collocation method for the fractional option pricing models. Rather than these methods, TFBSEs are numerically solved by other various methods, including, not limited to,  the meshless methods \cite{Phaochoo}, the spline interpolation method \cite{Ghafouri, Roul}, and numerical techniques with exponential convergence \cite{Hezhang}.  

In this work we consider the time fractional BS model \cite{Li, Chen} as follows:

\begin{equation}\label{eq1.1}
    \frac{\partial^\alpha Q(S,\tau)}{\partial \tau^\alpha} +\frac{1}{2} \sigma^2 S^2 \frac{\partial^2 Q(S,\tau)}{\partial S^2} + rS \frac{\partial Q(S,\tau)}{\partial S} - rQ(S,\tau)=0, \quad (S,\tau)=0 \in (0,\infty) \times (0,T), \tag{1.1}
\end{equation}
with initial and boundary conditions:
\[
Q(0,\tau) = f(\tau),  \quad Q(\infty,\tau) = g(\tau), \quad Q(S,T) = h(S),
\]
where \( 0 < \alpha \leq 1 \), \( \sigma \geq 0 \) is the volatility of the returns from the holding stock price \( S \), \( r \) is the risk-free rate and   \( T \) is the expiry time. The fractional derivative operator in  (\ref{eq1.1}) is a modified right Riemann–Liouville derivative which is defined as:

\begin{equation}\label{eq1.2}
    \frac{\partial^\alpha Q(S,\tau)}{\partial \tau^\alpha} = \frac{1}{\Gamma(1-\alpha)} \frac{d}{d\tau} \int_{\tau}^T\frac{Q(S,\xi) - Q(S,T)}{(\xi-\tau)^\alpha} d\xi, \quad 0 < \alpha < 1. \tag{1.2}
\end{equation}
For \( \alpha = 1 \), we get the classical BS model. Let \( \eta = T - \tau \), then for \( 0 < \alpha < 1 \), rewriting (\ref{eq1.2}), we have 
\begin{eqnarray*}
   \frac{\partial^\alpha Q(S,\tau)}{\partial \tau^\alpha} & = & \frac{1}{\Gamma(1-\alpha)} \frac{-d}{d\eta} \int_{T-\eta}^T \frac{Q(S,\xi) - Q(S,T)}{(\xi - (T-\eta))^\alpha} d\xi  \\
  & = &\frac{-1}{\Gamma(1-\alpha)} \frac{d}{d\eta} \int_0^\eta \frac{Q(S,T-\zeta) - Q(S,T)}{(\eta - \zeta)^\alpha} d\zeta.
\end{eqnarray*}
Moreover, defining $$ s = \ln S~\text{and}~u(s,\eta) = Q(e^s, T-\eta) ,$$ the model (\ref{eq1.1}) can be express as 
\begin{equation}\label{eq1.3}
    _0D^\alpha_{\eta} u(s,\eta) = \frac{1}{2} \sigma^2 \frac{\partial^2 u(s,\eta)}{\partial s^2} + \left( r - \frac{1}{2} \sigma^2 \right) \frac{\partial u(s,\eta)}{\partial s} - ru(s,\eta), \tag{1.3}
\end{equation}
with boundary conditions:
\[
u(-\infty, \eta) = f(\eta), \quad u(\infty, \eta) = g(\eta), \quad u(s,0) = h(s),
\]
where the fractional derivative is:
\begin{equation}\label{eq1.4}
    _0D^\alpha_{\eta} u(s,\eta) = \frac{1}{\Gamma(1-\alpha)} \frac{d}{d\eta} \int_0^\eta \frac{u(s,\eta) - u(s,0)}{(\eta - \zeta)^\alpha} d\zeta, \quad (0 < \alpha < 1). \tag{1.4}
\end{equation}
In order to solve the above model numerically it is necessary to truncate the original unbounded domain into a finite interval. For this, we restrict the range of variable s in problem (\ref{eq1.1}) to a finite interval($B_x, B_y$). Then the model  takes the following form:
\begin{equation}\label{eq1.5}
    _0D^\alpha_{\eta} u(s,\eta) = a \frac{\partial^2 u(s,\eta)}{\partial s^2} + b \frac{\partial u(s,\eta)}{\partial s} - cu(s,\eta)+f(s, \eta), \tag{1.5}
\end{equation}
with boundary conditions:
\[
u(B_x, \eta) = f(\eta), \quad u(B_y, \eta) = g(\eta), \quad u(s,0) = h(s),
\]
where 
\[ a =\frac{1}{2} \sigma^2 >0,\quad b=r-a, \quad c=r>0.  \]
Here, the source term $f(s,\eta)$ is added for numerical validation. The given model is an advection-diffusion-reaction model, and it is well known that when $a>0,\mbox{ } b<0,\mbox{ } c=0$ the model is a time fractional advection-diffusion model and $a>0,\mbox{ } b=0,\mbox{ } c\neq 0$ reaction-diffusion model. In this context of TFBSE, we employ the DQM method based on modified cubic B-splines in conjunction with the very convenient time-fractional discretization method L1 \cite{lin}. In essence, DQM, introduced by Bellman et al. \cite{Bellman}, approximates derivatives by expressing them as a weighted sum of function values at selected discrete node points within the problem domain. In DQM, various types of polynomials are used to obtain weights, including B-spline functions \cite{Dag, Karakoc}, cubic and modified B-spline functions\cite{Karakoc, Mital1, AliBashan, Bashan2, A.Babu}, quintic B-spline functions \cite{Bashan2}, Lagrange interpolation polynomials, Legendre polynomials, Hermite polynomials, and Sinc functions, among others. Ali Bashan et al. \cite{Bashan} used combination of Crank-Nicolson scheme and quintic B-spline based DQM to find the numerical solution of coupled KdV equation. In \cite{Bashan2}, DQM based on modified cubic B-spline was implemented to obtain the numerical solutions for the nonlinear Schrödinger (NLS) equation.  Jiwari et al. \cite{Jiwari1} used the combination of Lagrange interpolation method and a DQM based on modified set of cubic B-splines to find the numerical approximation of hyperbolic partial differential equations. In the works, a modified set of cubic B-splines are defined as a basis for the required polynomial space in a uniform grid. But this set of modified splines fails to preserve the optimal polynomial reproduction property near the boundary.  A. Babu et al. \cite{A.Babu} proposed a new modification of the standard cubic splines that retains the optimal accuracy near the boundary. The convenience and efficiency of these modified B-spline basis functions have led to the adoption of this approach to obtain the numerical solution of the TFBSE under study.

The rest of the content is organized as follows. Section 2 explains the discretization technique and develops an implicit numerical scheme. The theoretical stability is established in Section 3. In Section 4, we evaluate the proposed method using numerical examples. Section 5 concludes the article with highlights.

\section{Numerical Scheme}
This section explains the numerical scheme employed, beginning with the discretization process in the time and space derivatives. After the discretization process, we express the obtained implicit discretization scheme in its matrix form. Subsections follow that explain the calculation of weights associated with the DQM.
\subsection{Discretization}
Initially, we discretize the time-fractional derivative. For this partition the   domain $[0,T]$ uniformly into $N$ sub intervals, $0<\eta_0<\eta_1<\ldots , < \eta_N =T$, where $\Delta \eta=k=\eta_n - \eta_{n-1}=\dfrac{T}{N} \mbox{ for } n=1,2,3,\ldots , N.$  
Suppose $u(s,\eta) \in C^{(1)}$, for $0 < \alpha \leq 1$, the modified Riemann--Liouville derivative
\begin{eqnarray*}
    _{0} D_\eta^\alpha u(s, \eta) & = & \frac{1}{\Gamma(1-\alpha)} \frac{d}{d\eta} \int_0^\eta \frac{u(s,\zeta) - u(s,0)}{(\eta - \zeta)^\alpha} d\zeta \\
     & = &  \frac{1}{\Gamma(1-\alpha)} \frac{d}{d\eta} \left( \int_0^\eta \frac{u(s,\zeta)}{(\eta - \zeta)^\alpha} d\zeta - \int_0^\eta \frac{u(s,0)}{(\eta - \zeta)^\alpha} d\zeta \right)\\
     &= &  \frac{1}{\Gamma(1-\alpha)} \frac{d}{d\eta} \int_0^\eta \frac{u(s,\zeta)}{(\eta - \zeta)^\alpha} d\zeta - \frac{u(s,0)}{\Gamma(1-\alpha)} \frac{d}{d\eta} \int_0^\eta \frac{1}{(\eta - \zeta)^\alpha} d\zeta \\
     & = &  \frac{1}{\Gamma(1-\alpha)} \int_0^\eta \frac{\partial u(s,\zeta)}{\partial \zeta} (\eta - \zeta)^{-\alpha} d\zeta \\
     & = & {}_0^C D_\eta^\alpha u(s,\eta).
\end{eqnarray*}
Here the operator ${}_0^C D_\eta^\alpha u(s,\eta)$ is the Caputo derivative. Then by using the L1 discretization formula for Caputo fractional derivative the operator  $_{0} D_\eta^\alpha u(s, \eta)$   at point $(s_i, \eta_{k+1})$  can be approximated as:
\begin{align*}
_{0} D_\eta^\alpha u(s_i, \eta_{k+1}) &= \frac{1}{\Gamma(1-\alpha)} \int_0^{\eta_{k+1}} \frac{\partial u\left(s_i, \zeta\right)}{\partial \zeta}\left(\eta_{k+1}-\zeta\right)^{-\alpha} d \zeta \\
&= \frac{1}{\Gamma(1-\alpha)} \sum_{p=0}^k \int_{(p-1) k}^{p k}\left(\frac{u_i^p-u_i^{p-1}}{k}+\mathcal{O}(k)\right)((n+1) k-\zeta)^{-\alpha} d \zeta \\
&= \frac{1}{\Gamma(1-\alpha)} \sum_{p=0}^k \left(\frac{u_i^p-u_i^{p-1}}{k}+\mathcal{O}(k)\right) \cdot\\
& \hspace{1in}\left( \dfrac{((n+1)k-pk)^{1-\alpha}-((n+1)k-(p+1)k)^{1-\alpha}}{1-\alpha}\right) \\
_{0} D_\eta^\alpha u(s_i, \eta_{k+1}) &= \frac{1}{\Gamma(2-\alpha)} \sum_{p=0}^k \left(\frac{u_i^p-u_i^{p-1}}{k}+\mathcal{O}(k)\right) \cdot \\
& \hspace{1in} \left(((n+1)k-pk)^{1-\alpha}-((n+1)k-(p+1)k)^{1-\alpha}\right),
\end{align*}
on further simplification, we get the final L1 discretized formula as follows 
\begin{equation}\label{eq2.1}
     _{0} D_\eta^\alpha u(s_i, \eta_{k+1}) = \dfrac{\Delta \eta^{-\alpha}}{\Gamma (2-\alpha)} \sum_{p=0}^k \delta_p\left(u_i^{k+1-p}-u_i^{k-p}\right)+\mathcal{O}(\Delta \eta^{2-\alpha}), \tag{2.1}
\end{equation}
where, $u_i^k=u(s_i,\eta_k), \quad \delta_p = (p+1)^{1-\alpha} - p^{1-\alpha}.$ Moreover, the coefficients $\delta _p$ satisfies the following relations 
   \begin{itemize}
       \item[i.]  $\delta _0 =1$,
       \item [ii.] $\delta _p>0, \quad \mbox{ for every } 0\leq p \leq N$,
       \item [iii.] $\delta _{p-1}>\delta _ p, \mbox{ for every } 1\leq p \leq N.$
   \end{itemize}
Now we discretize the space derivatives by using differential quadrature method.   Let  $M\in \mathbb{Z}^+$,  partition the spatial part of the  domain  into $M$ subintervals. Let $a=s_{0}<s_{1}<\ldots <s_{M}=b$ be the corresponding  $M+1$ nodal points  with uniform  the step size $h=\dfrac{b-a}{M}$. Using the differential quadrature approximation, at every nodes  $s_{i},\: i=0,1,\ldots ,M ,$ to the spatial derivatives, the derivatives takes the form 
\begin{equation}\label{eq2.2}
u_{s}(s_{i}, \eta)= \sum_{j=0}^{M} p_{i j}^{(1)} u\left(s_{j}, \eta\right),\tag{2.2}
\end{equation}
\begin{equation}\label{eq2.3}
u_{ss}(s_{i}, \eta)= \sum_{j=0}^{M} p_{i j}^{(2)} u\left(s_{j}, \eta\right),\tag{2.3}
\end{equation}
where, the coefficients  $p_{i j}^{(1)}, p_{i j}^{(2)} ,\: i, j =0,1,\ldots , M$ are the weights for the approximation. The modified cubic B-spline functions constructed by A. Babu \cite{A.Babu} are used to  determine the weights. This methodology is discussed in the upcoming subsection.
 Combining the time discretization scheme (2.1) with spatial discretization (\ref{eq2.2}) and (\ref{eq2.3}) we can derive the full discretization scheme for (\ref{eq1.5}) as follows:
 
\begin{equation}\label{eq2.4}
  \dfrac{\Delta \eta^{-\alpha}}{\Gamma (2-\alpha)} \sum_{p=0}^k \delta_p\left(u_i^{k+1-p}-u_i^{k-p}\right)  = a \sum_{j=0}^{M} p_{i j}^{(2)} u_j^{k+1}+b\sum_{j=0}^{M} p_{i j}^{(1)} u_j^{k+1}+cu_i^{k+1}+f_i^{k+1}.\tag{2.4}
\end{equation}
Multiplying by $d=\Delta \eta^{\alpha}\Gamma (2-\alpha)$ on both sides we get
\begin{equation}\label{eq2.5}
   \sum_{p=0}^k \delta_p\left(u_i^{k+1-p}-u_i^{k-p}\right)  = ad \sum_{j=0}^{M} p_{i j}^{(2)} u_j^{k+1}+bd\sum_{j=0}^{M} p_{i j}^{(1)} u_j^{k+1}+cdu_i^{k+1}+df_i^{k+1}.\tag{2.5}
\end{equation}
Let \[\textbf{X}  =\begin{bmatrix}
p_{00}^{(1)} & p_{01}^{(1)} & \cdots & p_{0M}^{(1)} \\
p_{10}^{(1)} & p_{11}^{(1)} & \cdots & p_{1M}^{(1)} \\
\vdots & \vdots & \ddots & \vdots \\
p_{M0}^{(1)} & p_{M1}^{(1)} & \cdots & p_{MM}^{(1)} 
\end{bmatrix} \textnormal{be the weights corresponding to the first order derivative},
\]
\vspace{0.1in}
 \[\textbf{Y}  =\begin{bmatrix}
p_{00}^{(2)} & p_{01}^{(2)} & \cdots & p_{0M}^{(2)} \\
p_{10}^{(2)} & p_{11}^{(2)} & \cdots & p_{1M}^{(2)} \\
\vdots & \vdots & \ddots & \vdots \\
p_{M0}^{(2)} & p_{M1}^{(2)} & \cdots & p_{MM}^{(2)} 
\end{bmatrix} \textnormal{be the weights corresponding to the second order derivative},
\]
and $\textbf{I}$ be the $(M+1)\times (M+1)$ identity matrix, then we can represent the discretization scheme (\ref{eq2.5}) in matrix form as 

\begin{equation}\label{eq2.6}
     \textbf{LU}^{1}=\textbf{U}^{0}+d\textbf{F}^{1},
    \tag{2.6}
\end{equation}
for $k\geq 1$,
\begin{equation}\label{eq2.7}
     \textbf{LU}^{k+1}=  \sum_{p=0}^{k-1}(\delta_p - \delta_{p+1})\textbf{U}^{k-p} +\textbf{U}^{0}+d\textbf{F}^{k+1},
    \tag{2.7}
\end{equation}
where,
$$\displaystyle \textbf{L}= (1+cd)\textbf{I}- \frac{1}{h^2}ad\textbf{Y}-\frac{1}{h}bd\textbf{X},$$
\[
\begin{aligned}
   \textbf{U}^{k+1}& = [u_0,\quad u_1,\quad \ldots ,\quad u_M]^T ,\quad u_i =u(s_i, \eta_{k+1}),
\end{aligned}
\]
\[
\begin{aligned}
   \textbf{F}^{k+1}& = [f_0,\quad f_1,\quad \ldots ,\quad f_M]^T ,\quad f_i =f(s_i, \eta_{k+1}).
\end{aligned}
\]

\subsection{Modified Splines and Calculation of Weights}
This subsection deals with the cubic B-splines and their modifications to find the differential quadrature weights which is discussed in \cite{A.Babu}. For each $j\in \mathbb{Z}$ the standard cubic B-spline $C_j(s)$ which is symmetric about the node point $s_j$ is defined by

\[ 
C_j(s) := \frac{1}{h^3}
\begin{cases}
(s - s_{j-2})^3, & s \in [s_{j-2}, s_{j-1}), \\
(s - s_{j-2})^3 - 4(s - s_j)^3, & s \in [s_{j-1}, s_j), \\
(s_{j+2} - s)^3 - 4(s_{j+1} - s)^3, & s \in [s_j, s_{j+1}), \\
(s_{j+2} - s)^3, & s \in [s_{j+1}, s_{j+2}), \\
0, & \text{otherwise}.
\end{cases} \tag{2.8}
\]
Note that in the interval $[s_{j-2},\mbox{ } s_{j+2}]$,  $C_j$ is a twice continuously differentiable function. Table 1 explicitly gives the function values of $C_j$'s and its derivatives at each node point.
\begin{table}[H]
\centering
\renewcommand{\arraystretch}{1.8} % increase row height
\setlength{\tabcolsep}{20pt} % increase column spacing

\begin{tabular}{c|ccccc}
\toprule
$s$ & $s_{j-2}$ & $s_{j-1}$ & $s_j$ & $s_{j+1}$ & $s_{j+2}$ \\
\midrule
$C_j(x)$ & 0 & 1 & 4 & 1 & 0 \\
$C_j'(x)$ & 0 & $\dfrac{3}{h}$ & 0 & $-\dfrac{3}{h}$ & 0 \\
$C_j''(x)$ & 0 & $\dfrac{6}{h^2}$ & $-\dfrac{12}{h^2}$ & $\dfrac{6}{h^2}$ & 0 \\
\bottomrule
\end{tabular}
\caption{Derivatives of cubic splines at each node}
\end{table}We can observe that the cubic splines $C_{-1}, C_0, C_1, C_{M-1}, C_{M}, C_{M+1}$ are not fully supported in the spatial domain $[a,b]$. Therefore, such standard cubic splines near the boundary must be modified. An optimally accurate modification has been proposed in \cite{A.Babu} and is as follows:

\[
\begin{aligned} \displaystyle 
\tilde{C}_0 &:= C_0 + 4C_{-1}, \\
\tilde{C}_1 &:= C_1 - \tfrac{7}{2}C_{-1} + \tfrac{5}{8}C_0, \\
\tilde{C}_2 &:= C_2 + \tfrac{88}{37}C_1 - \tfrac{21}{37}C_0 - \tfrac{4}{37}C_1, \\
\tilde{C}_3 &:= C_3 - C_{-1} + \tfrac{1}{4}C_0 - \tfrac{1}{4}C_2, \\
\tilde{C}_j &:= C_j, \qquad \mbox{ for all } j=4,5,\ldots, M-4,
\end{aligned}
\qquad
\begin{aligned} \displaystyle 
\tilde{C}_M &:= C_M + 4C_{M+1}, \\
\tilde{C}_{M-1} &:= C_{M-1} - \tfrac{7}{2}C_{M+1} + \tfrac{5}{8}C_M, \\
\tilde{C}_{M-2} &:= C_{M-2} + \tfrac{88}{37}C_{M+1} - \tfrac{21}{37}C_M - \tfrac{4}{37}C_{M-1}, \\
\tilde{C}_{M-3} &:= C_{M-3} - C_{M+1} + \tfrac{1}{4}C_M - \tfrac{1}{4}C_{M-2}.
\end{aligned}
\]
By using this modified splines we get the matrix equation: $$\textbf{AX}^{T}=\textbf{B},$$ where, the matrix $\textbf{A} \mbox{ and } \textbf{B}$ are $(M+1) $ order matrices respectively given by 
\[
\mathbf{A}=\left[\begin{array}{cccccccccccc}
8 & 1 & 0 & 0 & 0 & 0 &0& \cdots &0& 0 & 0 & 0 \\
0 & \frac{37}{8} & 1 & 0 & 0 & 0 & 0&\cdots & 0&0 & 0 & 0 \\
0 & 0 & \frac{144}{37} & 1 & 0 & 0 & 0&\cdots &0& 0 & 0 & 0 \\
0 & 0 & 0 & \frac{15}{4} & 1 & 0 & 0&\cdots & 0&0 & 0 & 0 \\
0 & 0 & 0 & 1 & 4 & 1 & 0&\cdots &0& 0 & 0 & 0 \\
0 & 0 & 0 & 0& 1 & 4 & 1 & \cdots &0& 0 & 0 &0\\
\vdots & \vdots & \vdots & \vdots & \vdots& \ddots & \ddots & \ddots & \vdots & \vdots& \vdots & \vdots \\
0 & 0 & 0 &0&0& \cdots & 1 & 4 & 1 & 0 & 0 & 0 \\
0 & 0 & 0 &0&0& \cdots & 0 & 1 & \frac{15}{4} & 0 & 0 & 0 \\
0 & 0 & 0 & 0&0&\cdots & 0 & 0 & 1 & \frac{144}{37} & 0 & 0 \\
0 & 0 & 0 & 0&0&\cdots & 0 & 0 & 0 & 1 & \frac{37}{8} & 0 \\
0 & 0 & 0 &0&0& \cdots & 0 & 0 & 0 & 0 & 1 & 8
\end{array}\right],\]
\vspace{0.1in}
Here,  the   coefficient matrix  $\textbf{A}$ is  invertible, hence the weights associated with (\ref{eq2.2}) can be calculated  from the matrix equation: $$\textbf{X}^T=\textbf{A}^{-1}\textbf{B}.$$

\[
\mathbf{B}=\left[\begin{array}{cccccccccccc}
-12 & -3 & 0 & 0 & 0 & 0 &0 & \cdots& 0 & 0 & 0 & 0 \\
\frac{27}{2} & -\frac{15}{8} & -3 & 0 & 0 & 0 &0 & \cdots& 0 & 0 & 0 & 0 \\
-\frac{276}{37} & \frac{174}{37} & \frac{12}{37} & -3 & 0 & 0 & 0 &\cdots& 0 & 0 & 0 & 0 \\
3 & -\frac{3}{2} & 3 & \frac{3}{4} & -3 & 0 &0 & \cdots & 0& 0 & 0 & 0 \\
0 & 0 & 0 & 3 & 0 & -3 & 0 &\cdots& 0 & 0 & 0 & 0 \\
0& 0 & 0 & 0 & 3 & 0 & -3 & \cdots& 0 & 0 & 0 & 0 \\
\vdots & \vdots & \vdots & \vdots &\vdots & \ddots & \ddots & \ddots & \vdots & \vdots & \vdots&\vdots \\
0 & 0 & 0 & 0 &0 &\cdots & 3 & 0 & -3 & 0 & 0 & 0 \\
0 & 0 & 0 & 0 &0 &\cdots & 0 & 3 & -\frac{3}{4} & -3 & \frac{3}{2} & -3 \\
0 & 0 & 0 & 0 &0 &\cdots & 0 & 0 & 3 & -\frac{12}{37} & -\frac{174}{37} & \frac{276}{37} \\
0 & 0 & 0 &0 & 0 &\cdots & 0 & 0 & 0 & 3 & \frac{15}{8} & -\frac{27}{2} \\
0 & 0 & 0 & 0 &0 &\cdots & 0 & 0 & 0 & 0 & 3 & 12
\end{array}\right].
\]
 The weights corresponding to the second order derivative (\ref{eq2.3}) is given by the matrix $\textbf{Y}$ which can be find out by taking the square of the matrix  $\textbf{X}$, the weights corresponding to the first order derivative \cite{Shu}.
\section{Theoretical Stability}
Consider the discretization operator $\mathcal{I}_h$, where $h$ represents the mesh parameter involved. We examine the stability of the numerical scheme associated with $\mathcal{I}_h$ in the sense that it is stable whenever $\mathcal{I}_h^{-1}$ is uniformly bounded \cite{roger}. Namely, there exists a constant $\mathcal{C}$ independent of the mesh parameter representative $h$, with

$$
\left\|\mathcal{I}_h^{-1}\right\|_\infty \leq \mathcal{C}.
$$
The theorem \ref{theorem1} justifies the effective maximum norm stability of the scheme \ref{eq2.7}.
\begin{lemma}\label{lem1}
	Let $a, b, c$ and $d$ be the parameters defined in (\ref{eq1.5},\ref{eq2.5}), and let $$\mathbf{P}=\frac{a}{h^2} \mathbf{Y}+\frac{b}{h} \mathbf{X},$$
    where $h$, is the spacial mesh parameter. Then, $$
\|\mathbf{P}\|_{\infty} \leq \frac{|a|}{h^2} \cdot R_Y+\frac{|b|}{h} \cdot R_X, 
$$
for some positive constants $R_X$ and $R_Y$.
\end{lemma}
\begin{proof}
	Under the assumptions on $\mathbf A=(a_{ij}) \in \mathbb R^{M+1 \times M+1}$, suggested by Varah \cite{varah}, we can combine the bounds as follows:
$$
R_X=\frac{1}{\beta}\cdot \|\mathbf{B^T}\|_\infty \geq \left\|\mathbf{(A^T)}^{-1}\right\|_\infty\|\mathbf{B^T}\|_\infty \geq \left\|(\mathbf{A}^{-1} \mathbf{B})^T\right\|_\infty,
$$
where, $$\displaystyle \beta=\min_k\left( |a_{kk}|-\sum_{l\neq k}|a_{lk}| \right),$$ is the minimum dominance associated with the matrix $\mathbf{A^T}$, and $\|\mathbf{B^T}\|_\infty$ can be estimated numerically by finding the $1$-norm of $\mathbf{B}=(b_{ij}) \in \mathbb R^{M+1 \times M+1}$, for any size $M \geq 8$. Then, by submultiplicativity:
$$
R_Y=R_X^2 \geq \|\mathbf{X}\|_{\infty}^2 \geq \left\|\mathbf{X}^2\right\|_{\infty},
$$
and, consequently,
\begin{equation}\label{eq3.1}
\|\mathbf{P}\|_{\infty}=\left\|\frac{a}{h^2} \mathbf{Y}+\frac{b}{h} \mathbf{X}\right\|_{\infty} {\leq}\left\|\frac{a}{h^2} \mathbf{Y}\right\|_{\infty}+\left\|\frac{b}{h} \mathbf{X}\right\|_{\infty}=\frac{|a|}{h^2}\cdot R_Y+\frac{|b|}{h}\cdot R_X
 \tag{3.1}
\end{equation}

\end{proof}

\begin{lemma}\label{lem3}(Neumann Series Theorem \cite{nuemann}).
Let $\mathscr X$ be a Banach space, and $A \in C L(\mathscr X),$ the space of all continuous linear operators on $\mathscr X$, be such that $\|A\|_\infty<1$.
Then \\\\
(1) $I-A$ is invertible in $C L(\mathscr X)$,\\
(2) $\displaystyle (I-A)^{-1}=I+A+\cdots+A^n+\cdots=\sum_{n=0}^{\infty} A^n$,\\
(3) $\displaystyle \left\|(I-A)^{-1}\right\|_\infty \leqslant \frac{1}{1-\|A\|_\infty}$.\\\\
In particular, $I-A: \mathscr X \rightarrow \mathscr X$ is bijective: for each $y \in \mathscr X$, there exists a unique solution $x \in \mathscr X$ of the equation $x-A x=y$, and moreover,
$$
\|x\|_\infty \leq \left(\frac{1}{1-\|A\|_\infty}\right)\|y\|_\infty,
$$
so that $x$ depends continuously on $y$.
\end{lemma}
\begin{theorem}\label{theorem1}
Let $\mathbf L$ be the discrete operator defined by the equation \ref{eq2.7} as
\begin{equation}\label{eq3.2}
     \mathbf L(i)\textbf{U}^{k+1}=  \sum_{p=0}^{k-1}(\delta_p - \delta_{p+1})\textbf{U}^{k-p}(i) +\textbf{U}^{0}(i)+d\textbf{F}^{k+1}(i),
 \tag{3.2}
\end{equation}
where, $\mathbf L(i)$ and $\mathbf F(i)^{k+1}$, denote the $i$ th row of $\mathbf L$, and $\mathbf F^{k+1}$ respectively, where as $\textbf{U}^{k-p}(i)$ denotes the $i$ th coordinate of $\textbf{U}^{k-p}$. Then the discrete operator $\mathbf L$ satisfies the maximum norm estimate $$\|\mathbf L^{-1}\|_\infty \leq \mathcal C,$$ for some constant $\mathcal C>0$, provides an upper bound independent of mesh parameter $h$.
\end{theorem}

\begin{proof}
Let

$$
\mathbf{P} = \frac{a}{h^2} \mathbf{Y} + \frac{b}{h} \mathbf{X}.
$$
Then $\mathbf{L}$ can be expressed as
$$
\mathbf{L} = (1 + c d)\, \mathbf{I} - d\, \mathbf{P}.
$$
Now, as $c=r>0$, we factor out the scalar $1+cd$ to obtain,
$$
\mathbf{L} = (1 + c d) \left[ \mathbf{I} - {\frac{d}{1 + c d} \mathbf{P}} \right].
$$
Let $\displaystyle \mathbf{Q} = \frac{d}{1 + c d} \mathbf{P}$, so that $$\mathbf{L} = (1 + c d)(\mathbf{I} - \mathbf{Q}).$$ 
Subject to the sufficient conditions on mesh parameters $h$ and $\Delta \eta$ described by, 
$$d\|\mathbf P\|_\infty < 1+cd,$$ we ensure that matrix norm $\|\mathbf{Q}\|_\infty < 1$. Then, $\mathbf{I} - \mathbf{Q}$ is invertible by Lemma \ref{lem3}, and its inverse is given by the Neumann series:

$$
(\mathbf{I} - \mathbf{Q})^{-1} = \sum_{k=0}^{\infty} \mathbf{Q}^k.
$$
Moreover, we have the bound (by the Neumann series convergence estimate):
$$
\left\| (\mathbf{I} - \mathbf{Q})^{-1} \right\|_\infty \leq \frac{1}{1 - \|\mathbf{Q}\|_\infty}.
$$
It follows that $\mathbf{L}$ is invertible and its inverse is given by
$$
\mathbf{L}^{-1} = \frac{1}{1 + c d} (\mathbf{I} - \mathbf{Q})^{-1}.
$$
Taking norms, we obtain the bound
$$
\left\| \mathbf{L}^{-1} \right\|_\infty \leq \frac{1}{|1 + c d|} \left\| (\mathbf{I} - \mathbf{Q})^{-1} \right\|_\infty \leq \frac{1}{|1 + c d| - |d| \cdot \|\mathbf{P}\|_\infty} \leq\mathcal C.
$$
\end{proof}
\noindent Clearly, Theorem (\ref{theorem1}) holds for sufficiently smaller mesh parameters $h$. Therefore, the uniform boundedness of the inverse operator defined by
\ref{eq3.2} indicates the maximum norm stability of the numerical scheme \ref{eq2.7} under reasonable conditions on mesh parameters keeping arbitrary volatility and the rate of returns.

\section{Numerical Findings}
To illustrate the accuracy of the solution and the order of convergence of suggested numerical approach, two examples with exact solutions are provided in this section. Let $u_j$ and $\tilde{u}_j$ be the exact and numerical solution at the node point $j$ respectively.  The $L_2$ norm error is evaluated by using 
\begin{equation}\label{eq4.1}
    L_2 = \sqrt{ h \sum_{j=0}^{M} \left| u^{}_j - \tilde{u}_j \right|^2 }, \tag{4.1}
\end{equation}
and $L_{\infty}:$ the sup norm error by
  \begin{equation}\label{eq4.2}
    L_{\infty} = \max_{0 \leq j \leq M} \left| u_j - \tilde{u}_j \right|. \tag{4.2}
\end{equation}
Furthermore, for evaluating  the spatial order of convergence $OC_s$ 
 \begin{equation}\label{eq4.3}
    OC_s = \frac{\log(E^{h_1}/E^{h_2})}{\log(h_1/h_2)}, \tag{4.3}
\end{equation}
where, $E^{h_1}$ and $E^{h_2}$ represent the errors at space mesh sizes $h_1$ and $h_2$, respectively and for order of convergence $OC_\eta$ in time 
 \begin{equation}\label{eq4.4}
    OC_\eta = \frac{\log(E^{k_1}/E^{k_2})}{\log(k_1/k_2)},  \tag{4.4}
\end{equation}
where, $E^{k_1}$ and $E^{k_2}$ represent the errors at time step sizes $k_1$ and $k_2$, respectively are used.
\begin{example}
   \textnormal{Here, we  consider the time fractional BS equation which takes the following form:}

\begin{equation}\label{eq4.5}
\left\{\begin{array}{l}
{ }_0 D_\eta^\alpha u(s, \eta)=\displaystyle a \frac{\partial^2 u(s, \eta)}{\partial s^2}+b \frac{\partial u(s, \eta)}{\partial s}-c u(s, \eta)+f(s, \eta), \\
u(0, \eta)=0, \quad u(1, \eta)=0, \\
u(s, 0)=s^2(1-s),
\end{array}\right. \tag{4.5}
\end{equation}
\textnormal{where the source term} \\
\[ f(s, \eta)=\left(\frac{2 \eta^{2-\alpha}}{\Gamma(3-\alpha)}+\frac{2 \eta^{1-\alpha}}{\Gamma(2-\alpha)}\right) s^2(1-s)-(\eta+1)^2\left[a(2-6 s)+b\left(2 s-3 s^2\right)-c s^2(1-s)\right],\]
\textnormal{ is selected so that the exact solution  is given by \cite{Zhang, Yang},}
$$u(s, \eta)=(\eta+1)^2 s^2(1-s).$$
\textnormal{ The parameters values are considered as $$r=0.05, \quad \sigma=0.25, \quad a=\frac{1}{2} \sigma^2, \quad b=r-a, \quad c=r~\textnormal{ and  }~T=1.$$ } 
\end{example}
\noindent The numerical error values and spatial order of convergence ($OC_s$)   for $M=2^m\cdot 10, \quad N=10^{m+1},~m=0,1,2,...$, are displayed in  Table \ref{tab2}.

\begin{table}[H]
\centering
\begin{tabular}{c|c c c c|c c c c}
\toprule
\multirow{2}{*}{M} & \multicolumn{4}{c|}{$\alpha = 0.3$} & \multicolumn{4}{c}{$\alpha = 0.5$} \\ 
& $L_2$ Error & $OC_s$ & $L_{\infty}$ Error & $OC_s$ & $L_2$ Error & $OC_s$ & $L_{\infty}$ Error & $OC_s$ \\ 
\midrule
$2^0 \cdot 10$ & $1.330e{-3}$ &  & $7.293e{-4}$ & & $3.618e{-3}$ &  & $1.779e{-3}$ &  \\
$2^1\cdot 10$ & $4.079e{-5}$ & 5.02 & $1.723e{-5}$ & 5.40 & $1.666e{-4}$ & 4.44 & $6.259e{-5}$ & 4.82 \\
$2^2\cdot 10$ & $1.189e{-6}$ & 5.10 & $3.696e {-7}$ & 5.54 & $7.430e{-6}$ & 4.49 & $2.039e{-6}$ & 4.94 \\
$2^3\cdot 10$ & $3.421e{-8}$ & 5.12 & $5.191e {-9}$ & 6.15 & $3.307e{-7}$ & 4.49 & $6.566e{-8}$ & 4.96 \\
\bottomrule
\end{tabular}

\vspace{0.5cm}

\begin{tabular}{c|c c c c|c c c c}
\toprule
\multirow{2}{*}{M} & \multicolumn{4}{c|}{$\alpha = 0.7$} & \multicolumn{4}{c}{$\alpha = 0.9$} \\ 
& $L_2$ Error & $OC_s$ & $L_{\infty}$ Error & $OC_s$ & $L_2$ Error & Order & $L_{\infty}$ Error & $OC_s$ \\ 
\midrule
$2^0\cdot 10$ & $8.308e{-3}$ &  & $3.514e{-3}$ &  & $1.758e{-2}$ &  & $7.813e{-3}$ &  \\
$2^1\cdot 10$ & $5.945e{-4}$ & 3.80 & $1.846e{-4}$ & 4.25 & $1.975e{-3}$ & 3.15 & $6.415e{-4}$ & 3.61 \\
$2^2\cdot 10$ & $4.175e{-5}$ & 3.83 & $9.357e{-6}$ & 4.30 & $2.202e{-4}$ & 3.16 & $5.115e{-5}$ & 3.65 \\
$2^3\cdot 10$ & $2.939e{-6}$ & 3.83 & $4.696e{-7}$ & 4.32 & $2.463e{-5}$ & 3.16 & $4.064e{-6}$ & 3.65 \\
\bottomrule
\end{tabular}
\caption{DQM Error values and $OC_s$ for different $\alpha$ values of Example 4.1}
\label{tab2}
\end{table}
\noindent The tabular data demonstrates that the scheme exhibits high order accuracy in both $L_2$ and $L_\infty$ norms. Notably, there is an enhancement in the order of spatial convergence whenever $\alpha$ tends to 0 which is due to the memory effects in the fractional model. Table \ref{tab3} displays the error values and order of convergence in time for $M=80,\quad N=2^m\cdot 10,~m=0,1,2,...$. The data gives the order of convergence in time is $2-\alpha$.  The surface plot of the solutions when $\alpha =0.7$, $M=20$ is given in Figure \ref{Surfacelpot1}. Figure \ref{payoff1}: (a)-(d) depicts the plot of the numerical and exact solution at the payoff for different values of $\alpha$ when $M=20$. Moreover, H. Zhang\cite{Zhang} formed an implicit discrete scheme for the considered  model by the combination of finite difference method(FDM) to the space derivatives and classical L1 scheme for the time fractional derivative. The numerical data in \cite{Zhang} shows that the scheme has spatial order of convergence 2. The numerical error values  for this FDM scheme \cite{Zhang} for different values of $\alpha$ at $M=2^m\cdot 10, \quad N=10^{m+1},~m=0,1,2,...$, are displayed in  Table \ref{tab2(b)}. The mixed alternate segment Crank-Nicolson (MASC-N)  scheme developed by X. Yang\cite{Yang}  has spatial order of convergence  2. On Comparing with these techniques, the results shows the potential of modified  B-spline based DQM  for solving fractional differential equations with precision and stability.

\begin{figure}[H]
    \centering
    \begin{subfigure}[b]{0.45\textwidth}
        \centering
        \includegraphics[width=\textwidth]{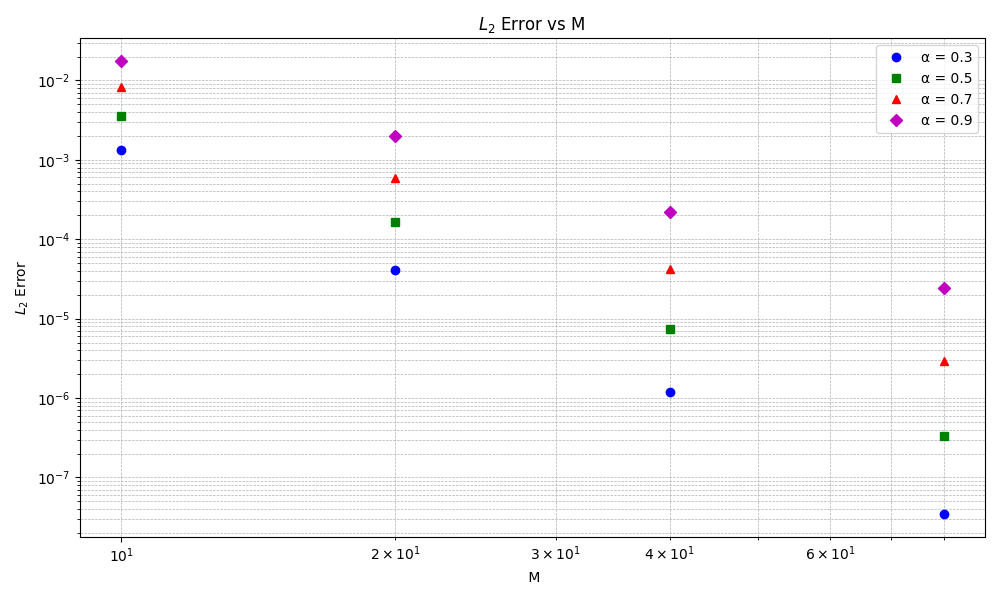}
        \caption{}
    \end{subfigure}
    \hspace{1cm}
    \begin{subfigure}[b]{0.45\textwidth}
        \centering
        \includegraphics[width=\textwidth]{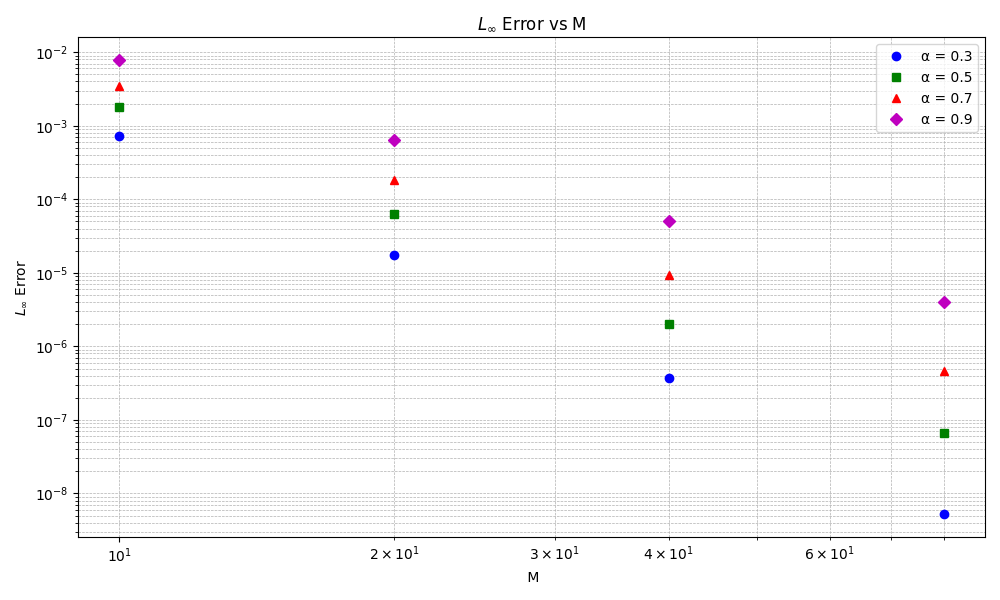}
         \caption{}
    \end{subfigure}
    \caption{(a)-(b) Error for different values of $\alpha$ at $M=20$ of Example 4.1}
    \label{error1}
\end{figure}

\begin{table}[H]
\centering
\begin{tabular}{c|c c c c|c c c c}
\toprule
\multirow{2}{*}{N} & \multicolumn{4}{c|}{$\alpha = 0.3$} & \multicolumn{4}{c}{$\alpha = 0.5$} \\ 
& $L_2$ Error & $OC_\eta$ & $L_{\infty}$ Error & $OC_\eta$ & $L_2$ Error & $OC_\eta$ & $L_{\infty}$ Error & $OC_\eta$ \\ 
\midrule
$2^0 \cdot 10$ & $3.375e{-3}$ &  & $7.293e{-4}$ & & $9.372e{-3}$ &  & $1.779e{-3}$ &  \\
$2^1\cdot 10$ & $1.100e{-3}$ & 1.62 & $2.404e{-4}$ & 1.60 & $3.436e{-3}$ & 1.44 & $6.598e{-4}$ & 1.43 \\
$2^2\cdot 10$ & $3.540e{-4}$ & 1.64 & $7.789e {-5}$ & 1.63 & $1.244e{-3}$ & 1.47 & $2.408e{-6}$ & 1.45 \\
$2^3\cdot 10$ & $1.127e{-4}$ & 1.65 & $2.492e {-5}$ & 1.64 & $4.469e{-4}$ & 1.48 & $8.701e {-5}$ & 1.47 \\
\bottomrule
\end{tabular}

\vspace{0.5cm}

\begin{tabular}{c|c c c c|c c c c}
\toprule
\multirow{2}{*}{N} & \multicolumn{4}{c|}{$\alpha = 0.7$} & \multicolumn{4}{c}{$\alpha = 0.9$} \\ 
& $L_2$ Error & $OC_\eta$ & $L_{\infty}$ Error & $OC_\eta$ & $L_2$ Error & $OC_\eta$ & $L_{\infty}$ Error & $OC_\eta$ \\ 
\midrule
$2^0\cdot 10$ & $6.142e{-1}$ &  & $1.221e{-1}$ &  & $1.38e{0}$ &  & $2.656e{-1}$ &  \\
$2^1\cdot 10$ & $2.567e{-1}$ & 1.26 & $5.099e{-2}$ & 1.26 & $6.601e{-1}$ & 1.06 & $1.261e{-1}$ & 1.07 \\
$2^2\cdot 10$ & $1.068e{-1}$ & 1.27 & $2.113e{-2}$ & 1.27 & $3.109e{-1}$ & 1.09 & $5.933e{-2}$ & 1.09 \\
$2^3\cdot 10$ & $4.365e{-2}$ & 1.29 & $8.639e {-3}$ & 1.29 & $1.459e{-1}$ & 1.09 & $2.782e{-2}$ & 1.09 \\
\bottomrule
\end{tabular}
\caption{Error values and $OC_\eta$ for different $\alpha$ values of Example 4.1}
\label{tab3}
\end{table}

\begin{table}[H]
\centering
\begin{tabular}{c|c c c c|c c c c}
\toprule
\multirow{2}{*}{M} & \multicolumn{4}{c|}{$\alpha = 0.3$} & \multicolumn{4}{c}{$\alpha = 0.5$} \\ 
& $L_2$ Error & $OC_s$ & $L_{\infty}$ Error & $OC_s$ & $L_2$ Error & $OC_s$ & $L_{\infty}$ Error & $OC_s$ \\ 
\midrule
$2^0 \cdot 10$ & $3.523e{-4}$ &  & $2.085e{-4}$ & & $2.605e{-3}$ &  & 1.284$e{-3}$ &  \\
$2^1\cdot 10$ & $4.799e{-4}$ & -0.44 & $1.346e{-4}$ & 0.631 & $3.296e{-4}$ & 2.98 & $9.477e{-5}$ & 3.76 \\
$2^2\cdot 10$ & $1.823e{-4}$ & 1.39 & $3.635e {-5}$ & 1.89 & $1.650e{-4}$ & 0.99 & $3.256e{-5}$ & 1.54 \\
$2^3\cdot 10$ & $6.486e{-5}$ & 1.49 & $9.147e {-6}$ & 1.99 & $6.059e{-5}$ & 1.45 & $8.483e{-6}$ & 1.94 \\
\bottomrule
\end{tabular}

\vspace{0.5cm}

\begin{tabular}{c|c c c c|c c c c}
\toprule
\multirow{2}{*}{M} & \multicolumn{4}{c|}{$\alpha = 0.7$} & \multicolumn{4}{c}{$\alpha = 0.9$} \\ 
& $L_2$ Error & $OC_s$ & $L_{\infty}$ Error & $OC_s$ & $L_2$ Error & Order & $L_{\infty}$ Error & $OC_s$ \\ 
\midrule
$2^0\cdot 10$ & $3.523e{-4}$ &  & $2.084e{-4}$ &  & $1.812e{-2}$ &  & $8.538e{-3}$ &  \\
$2^1\cdot 10$ & $1.979e{-4}$ & 0.83 & $7.503e{-5}$ & 1.47 & $1.611e{-3}$ & 3.49 & $5.509e{-4}$ & 3.95 \\
$2^2\cdot 10$ & $1.191e{-4}$ & 0.73 & $2.363e{-5}$ & 1.67 & $9.230e{-5}$ & 4.12 & $2.480e{-5}$ & 4.47 \\
$2^3\cdot 10$ & $5.352e{-5}$ & 1.15 & $7.385e{-6}$ & 1.67 & $2.806e{-5}$ & 1.71 & $4.436e{-6}$ & 2.48 \\
\bottomrule
\end{tabular}
\caption{FDM Error values and $OC_s$ for different $\alpha$ values of Example 4.1}
\label{tab2(b)}
\end{table}

\begin{example}
 \textnormal{In this example we  consider the time fractional BS equation with non homogeneous boundary condition  having the form:}
\begin{equation}\label{eq4.6}
\left\{\begin{array}{l} \displaystyle 
{ }_0 D_\eta^\alpha u(s, \eta)=a \frac{\partial^2 u(s, \eta)}{\partial s^2}+b \frac{\partial u(s, \eta)}{\partial s}-c u(s, \eta)+f(s, \eta), \\
u(0, \eta)=(\eta +1)^2, \quad u(1, \eta)=3(\eta +1)^2, \\
u(s, 0)=s^3+ s^2+ 1,
\end{array}\right. \tag{4.6}
\end{equation}
\textnormal{where the source term $f(s, \eta)$ is } \\
\[ \left(\frac{2 \eta^{2-\alpha}}{\Gamma(3-\alpha)}+\frac{2 \eta^{1-\alpha}}{\Gamma(2-\alpha)}\right) (s^3+ s^2+ 1)-(\eta+1)^2\left[a(6s+2)+b\left(2 s+3 s^2\right)-c (s^3+ s^2+ 1)\right],\]
\textnormal{is chosen so that the exact solution  is given by \cite{Zhang, Yang}} $$u(s, \eta)=(\eta+1)^2 (s^3+ s^2+ 1).$$
\textnormal{Here we take the parameters values as $$r=0.5, \quad a=1, \quad b=r-a, \quad c=r~\text{and,}~T=1. $$} 
\end{example}
\begin{figure}[H]
    \centering
        \includegraphics[width=\textwidth]{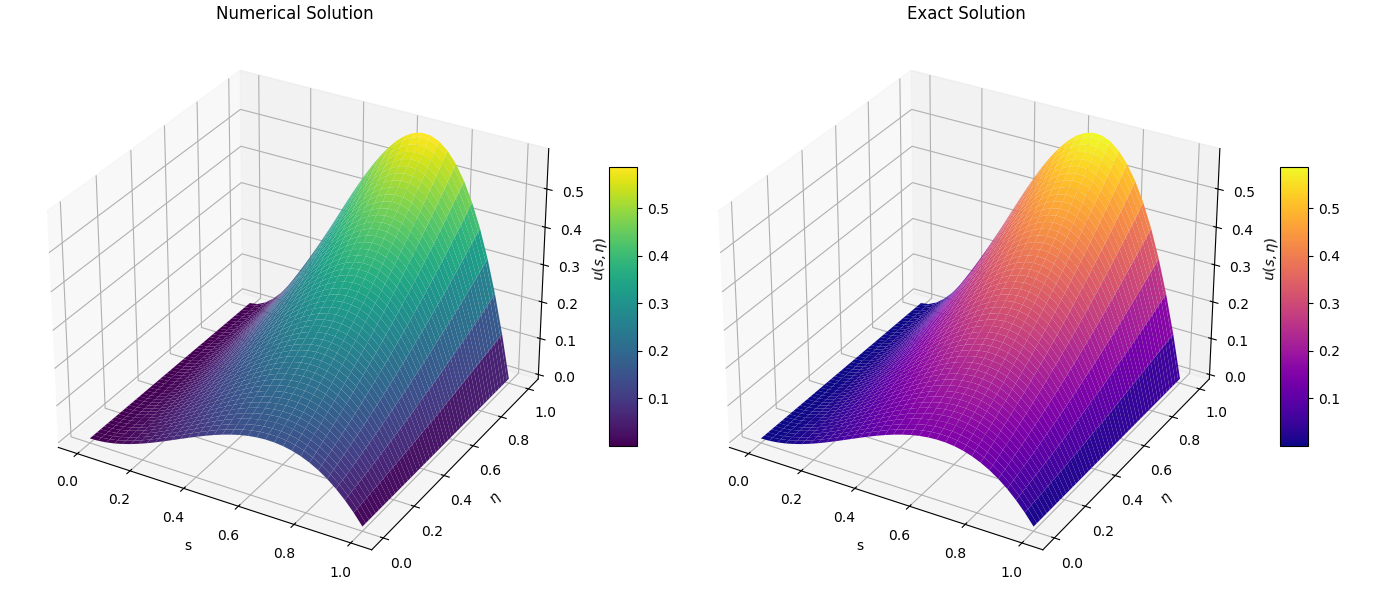}
    \caption{ Surface plot of numerical and exact solution at $\alpha =0.7, M=20$ of Example 4.1}
    \label{Surfacelpot1}
\end{figure} 

\begin{figure}[H]
    \centering
    \begin{subfigure}[b]{0.45\textwidth}
        \centering
        \includegraphics[width=\textwidth]{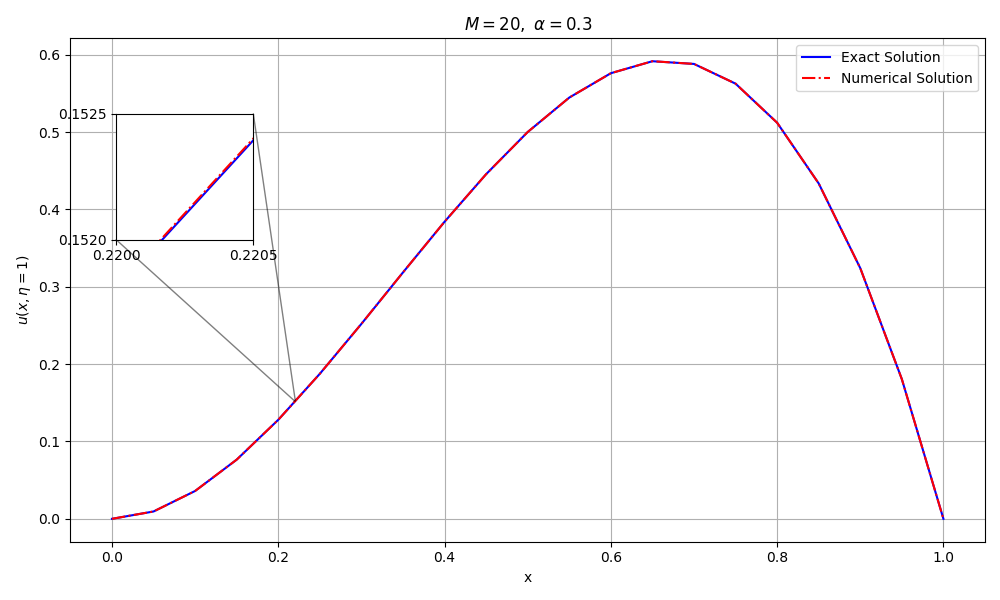}
        \caption{}
    \end{subfigure}
    \hspace{1cm}
    \begin{subfigure}[b]{0.45\textwidth}
        \centering
        \includegraphics[width=\textwidth]{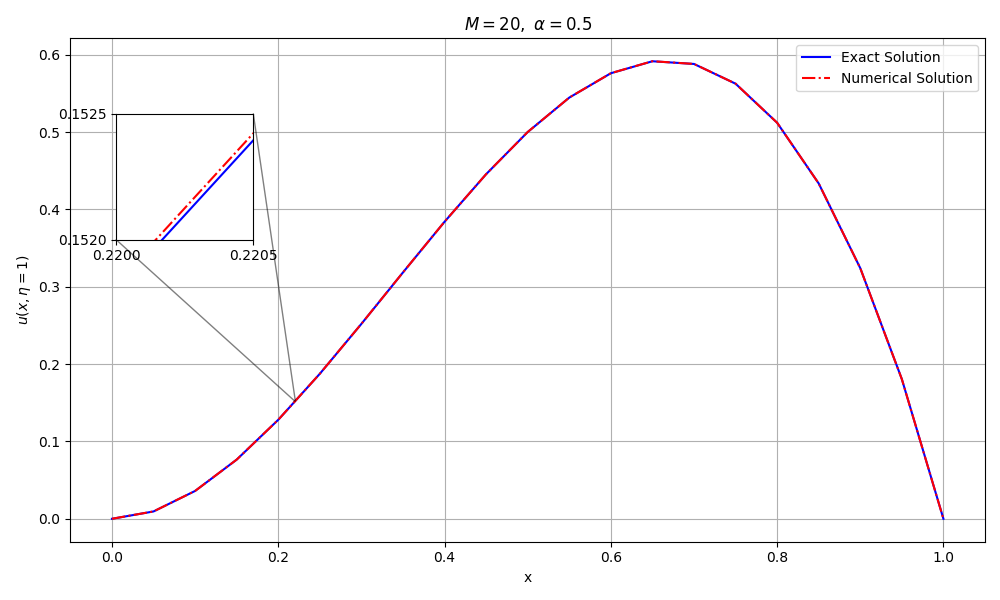}
         \caption{}
    \end{subfigure}
    \begin{subfigure}[b]{0.45\textwidth}
        \centering
        \includegraphics[width=\textwidth]{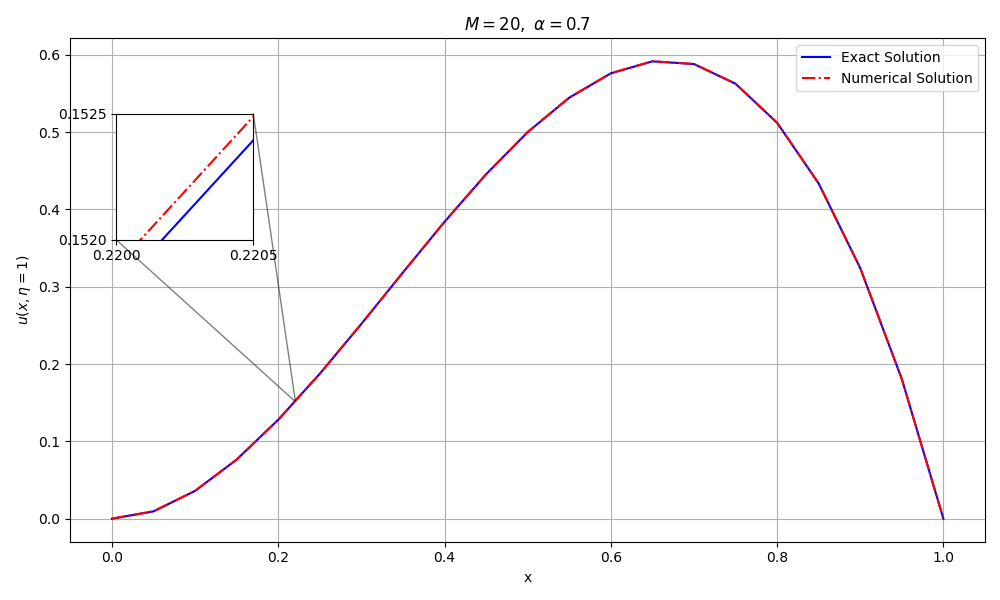}
         \caption{}
    \end{subfigure}
    \hspace{1.1cm}
    \begin{subfigure}[b]{0.45\textwidth}
        \centering
        \includegraphics[width=\textwidth]{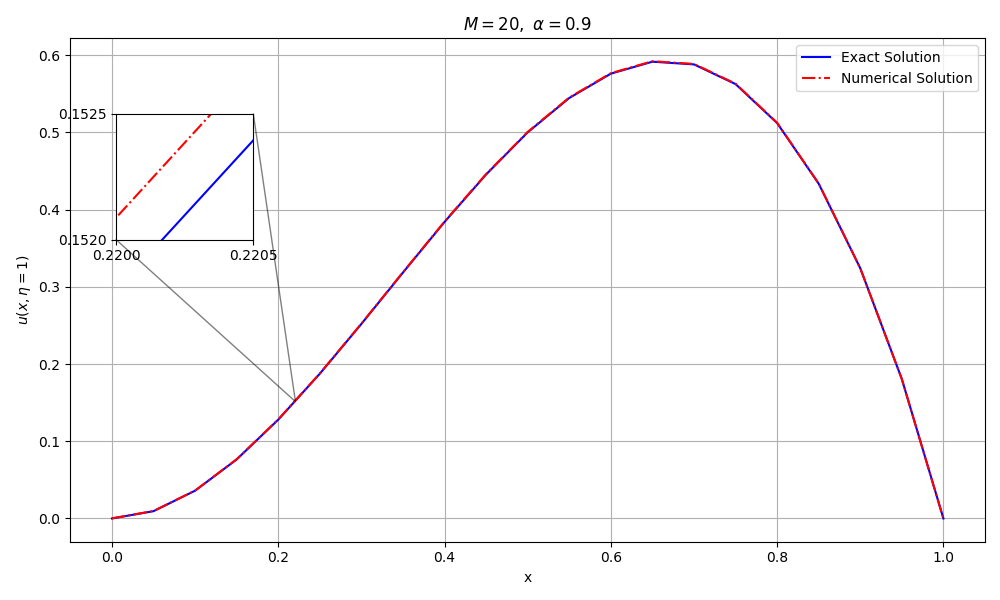}
         \caption{}
    \end{subfigure}
    \caption{(a)-(d) Comparison of Numerical vs Exact solution for different values of $\alpha$ at $M=20$ of Example 4.1}
    \label{payoff1}
\end{figure}

\noindent The evaluated numerical error  values and spatial order of convergence  ($OC_s$)   for $M=2^m\cdot 10, \quad N=10^{m+1}, m=0,1,2,... $ are displayed in  Table \ref{tab4}. The tabulated results indicate that for all values of $\alpha$, the errors in both norms decreases as $M$ increases, indicating that the numerical method converges with refinement. 
\begin{table}[H]
\centering
\begin{tabular}{c|c c c c|c c c c}
\toprule
\multirow{2}{*}{M} & \multicolumn{4}{c|}{$\alpha = 0.3$} & \multicolumn{4}{c}{$\alpha = 0.5$} \\ 
& $L_2$ Error & $OC_s$ & $L_{\infty}$ Error & $OC_s$ & $L_2$ Error & $OC_s$ & $L_{\infty}$ Error & $OC_s$ \\ 
\midrule
$2^0 \cdot 10$ & $3.104e{-2}$ &  & $1.737e{-2}$ & & $9.252e{-2}$ &  & $5.034e{-2}$ &  \\
$2^1\cdot 10$ & $9.560e{-4}$ & 5.02 & $3.975e{-4}$ & 5.45 & $4.268e{-3}$ & 4.44 & $1.719e{-3}$ & 4.87 \\
$2^2\cdot 10$ & $2.116e{-5}$ & 5.46 & $6.643e{-6}$ & 5.90 & $1.728e{-4}$ & 4.63 & $5.242e{-5}$ & 5.03 \\
\bottomrule
\end{tabular}

\vspace{0.5cm}

\begin{tabular}{c|c c c c|c c c c}
\toprule
\multirow{2}{*}{M} & \multicolumn{4}{c|}{$\alpha = 0.7$} & \multicolumn{4}{c}{$\alpha = 0.9$} \\ 
& $L_2$ Error & $OC_s$ & $L_{\infty}$ Error & $OC_s$ & $L_2$ Error & Order & $L_{\infty}$ Error & $OC_s$ \\ 
\midrule
$2^0\cdot 10$ & $2.328e{-1}$ &  & $1.222e{-1}$ &  & $7.109e{-1}$ &  & $2.656e{-1}$ &  \\
$2^1\cdot 10$ & $1.678e{-2}$ & 3.79 & $6.466e{-3}$ & 4.24 & $5.897e{-2}$ & 3.59 & $2.182e{-2}$ & 3.60 \\
$2^2\cdot 10$ & $1.188e{-3}$ & 3.82 & $3.287e{-4}$ & 4.07 & $6.542e{-3}$ & 3.17 & $1.742e{-3}$ & 3.64 \\
\bottomrule
\end{tabular}
\caption{Error and $OC_S$ for different $\alpha$ values of Example 4.2}
\label{tab4}
\end{table}
\noindent Table \ref{tab5} displays the error values and the order of convergence in time for $M=80,\quad N=2^m\cdot 10,~m=0,1,2,...$, and the data show that the order of convergence in time is $2-\alpha$.
\begin{table}[H]
\centering
\begin{tabular}{c|c c c c|c c c c}
\toprule
\multirow{2}{*}{N} & \multicolumn{4}{c|}{$\alpha = 0.3$} & \multicolumn{4}{c}{$\alpha = 0.5$} \\ 
& $L_2$ Error & $OC_\eta$ & $L_{\infty}$ Error & $OC_\eta$ & $L_2$ Error & $OC_\eta$ & $L_{\infty}$ Error & $OC_\eta$ \\ 
\midrule
$2^0 \cdot 10$ & $7.894e{-2}$ &  & $1.709e{-2}$ & & $2.415e{-1}$ &  & $5.015e{-2}$ &  \\
$2^1\cdot 10$ & $2.470e{-2}$ & 1.68 & $5.442e{-3}$ & 1.65 & $9.045e{-2}$ & 1.42 & $1.862e{-2}$ & 1.43 \\
$2^2\cdot 10$ & $6.984e{-3}$ & 1.82 & $1.607e {-3}$ & 1.75 & $3.120e{-2}$ & 1.54 & $6.573e{-3}$ & 1.50 \\
\bottomrule
\end{tabular}
\vspace{0.5cm}

\begin{tabular}{c|c c c c|c c c c}
\toprule
\multirow{2}{*}{N} & \multicolumn{4}{c|}{$\alpha = 0.7$} & \multicolumn{4}{c}{$\alpha = 0.9$} \\ 
& $L_2$ Error & $OC_\eta$ & $L_{\infty}$ Error & $OC_\eta$ & $L_2$ Error & Order & $L_{\infty}$ Error & $OC_\eta$ \\ 
\midrule
$2^0\cdot 10$ & $6.142e{-1}$ &  & $1.221e{-1}$ &  & $1.38e{0}$ &  & $2.656e{-1}$ &  \\
$2^1\cdot 10$ & $2.567e{-1}$ & 1.26 & $5.099e{-2}$ & 1.27 & $6.601e{-1}$ & 1.06 & $1.261e{-2}$ & 1.07 \\
$2^2\cdot 10$ & $1.068e{-1}$ & 1.27 & $2.113e{-2}$ & 1.27 & $3.109e{-1}$ & 1.09 & $5.933e{-2}$ & 1.09 \\
\bottomrule
\end{tabular}

\caption{Error and $OC_\eta$ for different $\alpha$ values of Example 4.2}
\label{tab5}
\end{table}

\begin{comment}
\vspace{0.5cm}
    \begin{tabular}{c|c c c c|c c c c}
\toprule
\multirow{2}{*}{M} & \multicolumn{4}{c|}{$\alpha = 0.7$} & \multicolumn{4}{c}{$\alpha = 0.9$} \\ 
& $L_2$ Error & $OC_s$ & $L_{\infty}$ Error & $OC_s$ & $L_2$ Error & Order & $L_{\infty}$ Error & $OC_s$ \\ 
\midrule
$2^0\cdot 10$ & $8.308e^{-3}$ &  & $3.514e^{-3}$ &  & $1.758e{-2}$ &  & $7.813e^{-3}$ &  \\
$2^1\cdot 10$ & $5.945e^{-4}$ & 3.80 & $1.846e^{-4}$ & 4.25 & $1.975e^{-3}$ & 3.15 & $6.415e^{-4}$ & 3.61 \\
$2^2\cdot 10$ & $4.175e^{-5}$ & 3.83 & $9.357e^{-6}$ & 4.30 & $2.202e^{-4}$ & 3.16 & $5.115e^{-5}$ & 3.65 \\
$2^3\cdot 10$ & $2.939e^{-6}$ & 3.83 & $4.696e ^{-7}$ & 4.32 & $2.463e^{-5}$ & 3.16 & $4.064e^{-6}$ & 3.65 \\
\bottomrule
\end{tabular}
\end{comment}
\begin{figure}[H]
    \centering
    \begin{subfigure}[b]{0.45\textwidth}
        \centering
        \includegraphics[width=\textwidth]{e1errorl2.png}
        \caption{}
    \end{subfigure}
    \hspace{1cm}
    \begin{subfigure}[b]{0.45\textwidth}
        \centering
        \includegraphics[width=\textwidth]{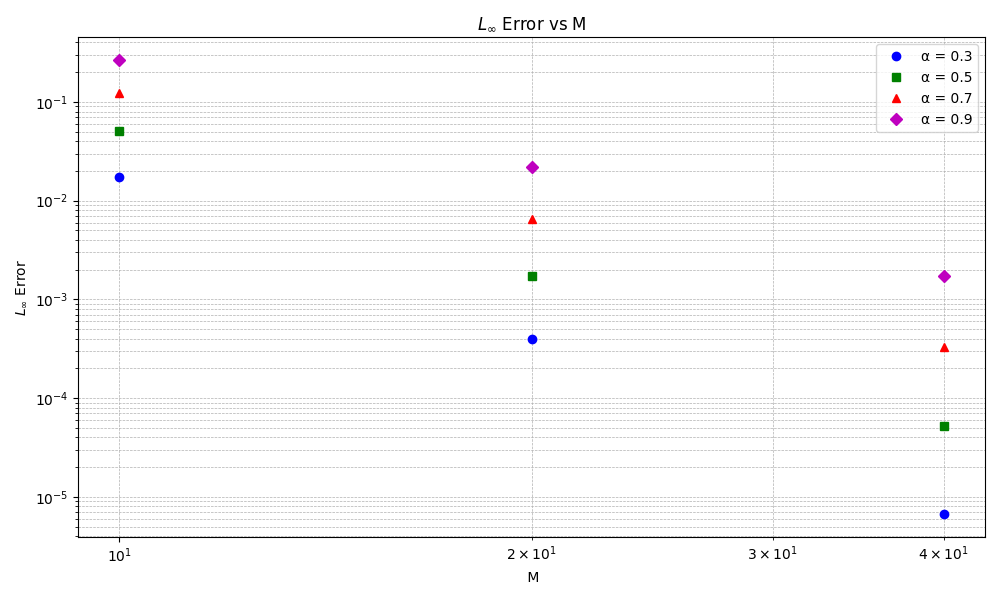}
         \caption{}
    \end{subfigure}
    \caption{(a)-(b) Error for different values of $\alpha$ at $M=20$ of Example 4.2}
    \label{error1}
\end{figure}
\noindent Figure \ref{fig3}: (a)-(d) presents the comparison of numerical and exact solution for various values of $\alpha$ for $M=20$.

\begin{figure}[H]
    \centering
    \begin{subfigure}[b]{0.45\textwidth}
        \centering
        \includegraphics[width=\textwidth]{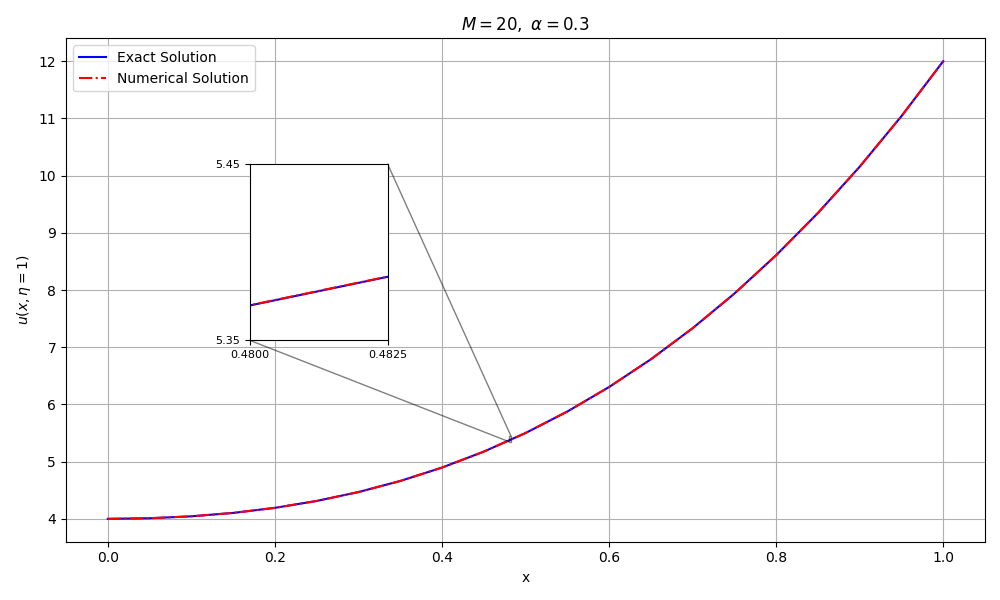}
        \caption{}
    \end{subfigure}
    \hspace{1cm}
    \begin{subfigure}[b]{0.45\textwidth}
        \centering
        \includegraphics[width=\textwidth]{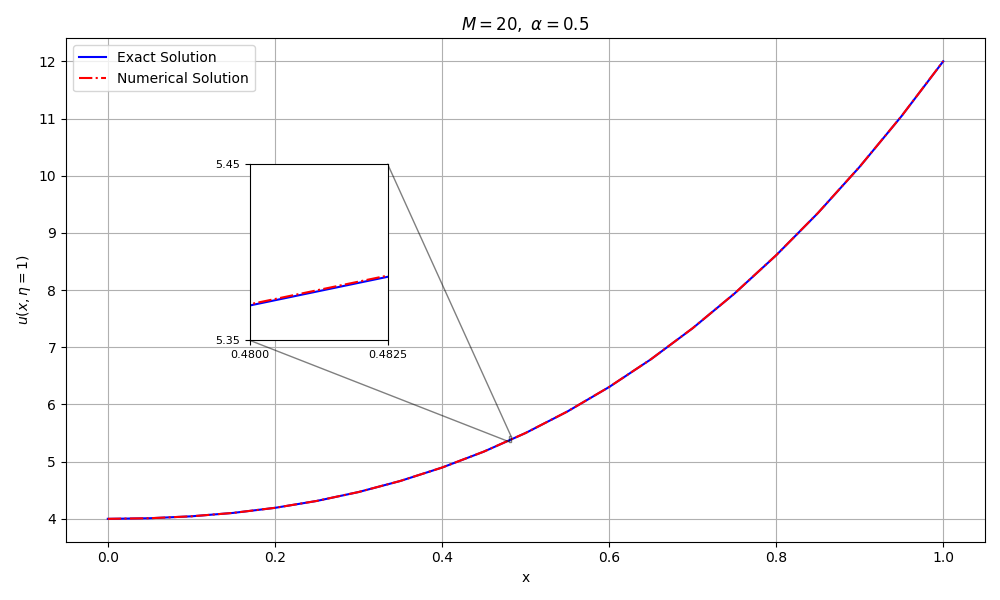}
         \caption{}
    \end{subfigure}
    \begin{subfigure}[b]{0.45\textwidth}
        \centering
        \includegraphics[width=\textwidth]{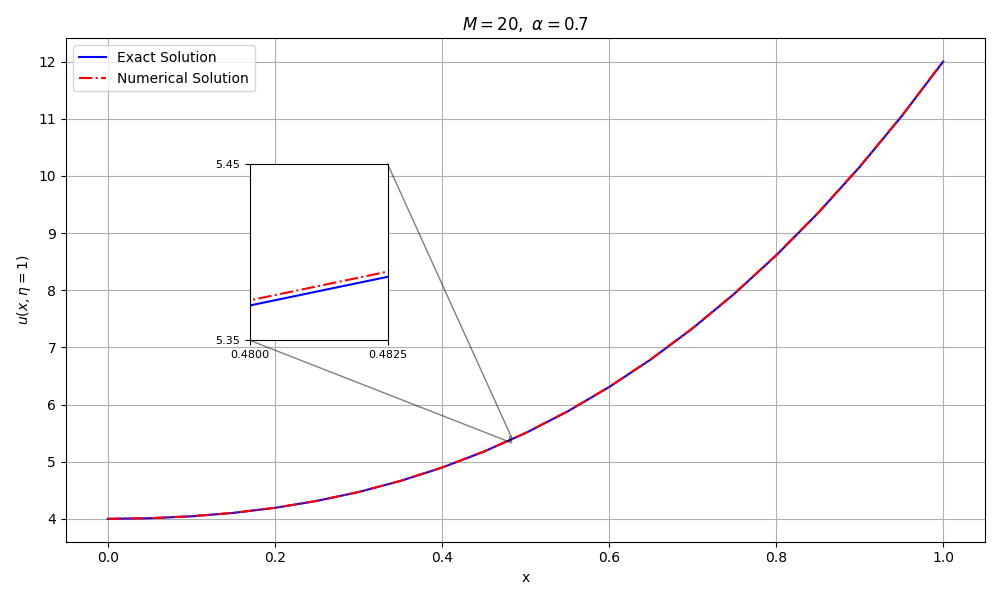}
         \caption{}
    \end{subfigure}
    \hspace{1.1cm}
    \begin{subfigure}[b]{0.45\textwidth}
        \centering
        \includegraphics[width=\textwidth]{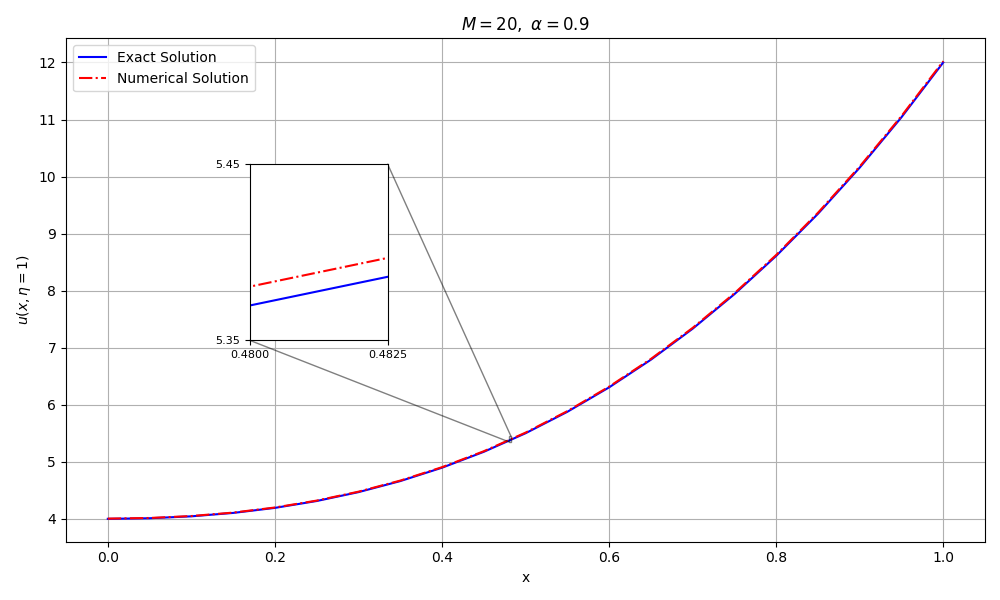}
         \caption{}
    \end{subfigure}
    \caption{(a)-(d) Comparison of Numerical vs Exact solution for different values of $\alpha$ at $M=20$ of Example 4.2}
    \label{fig3}
\end{figure}
\noindent Figure \ref{fig4} presents the surface plot of the solutions for $\alpha =0.7$ and $M=20$. 
\begin{figure}[H]
    \centering
        \includegraphics[width=\textwidth]{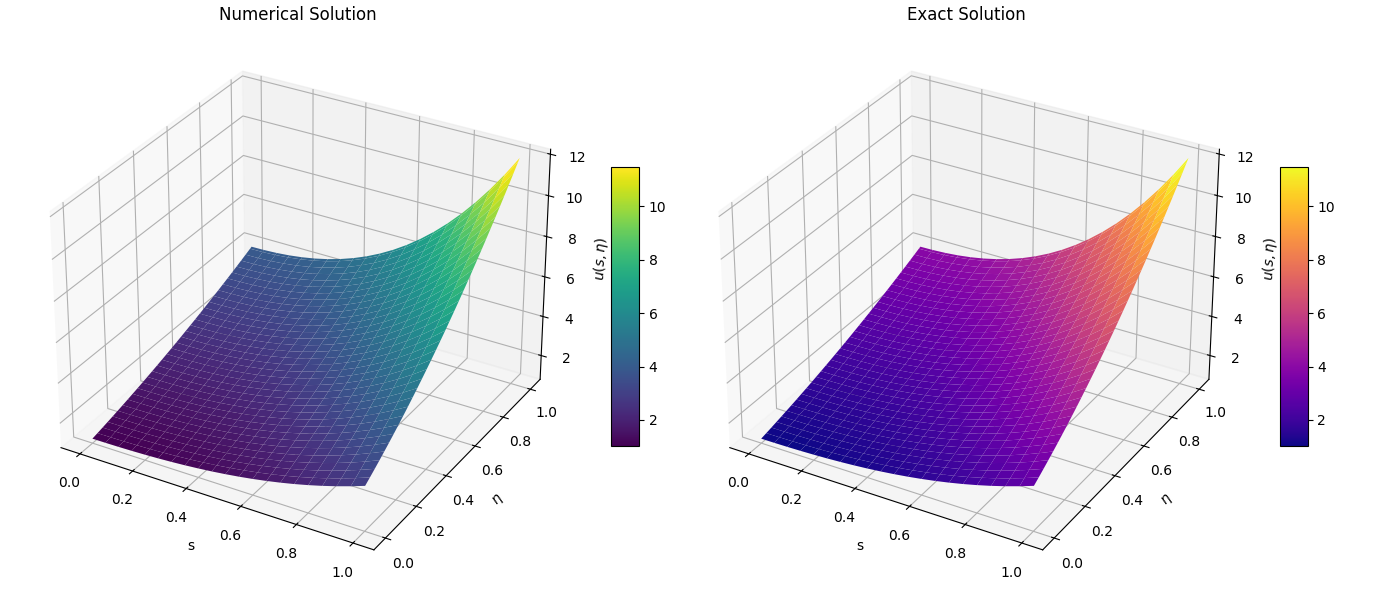}
    \caption{ Surface plot of numerical and exact solution at $\alpha =0.7, M=20$ of Example 4.2}
    \label{fig4}
\end{figure}

We observe that from the Table \ref{tab4}, for $\alpha$ tending to zero, the order of convergence is relatively high, shows the memory effect of the model, when $\alpha$ goes to 1 the order of convergence is near to 4, compared with the existing method in \cite{Zhang}.
\section{Conclusion}
This paper constructs an implicit numerical scheme to solve TFBSE. The construction of the numerical algorithm is achieved by discretizing the time-fractional derivative by the L1 finite difference method and the spatial derivative by the modified cubic B-spline-based differential quadrature method on uniform meshes. The stability of the proposed method has been studied by estimating an upper bound for the maximum norm of the inverse operator. Thanks to the Neumann series theorem it provides a uniform bound for the inverse operator under reasonable conditions
on the mesh parameters. The performance of the demonstrated method is tested on two sets of examples having exact solutions.   The results show that the numerical method exhibits a fourth-order convergence in the space direction and the order $2-\alpha$ in time. Moreover, we observe an enhancement in order of spatial convergence whenever $\alpha$ tends to $0$. The computational results of these experiments are compared with some existing popular techniques which suggest that the proposed method outperforms conventional methods from the literature in terms of solution accuracy. 
\section*{Acknowledgment}
The authors express their gratitude to the anonymous referees for their insightful input, valuable comments, and suggestions. The authors also thank the University Grants Commission (UGC), India, for the financial aid provided for this research.
\bibliographystyle{elsarticle-num}


\end{document}